\documentclass[a4paper]{amsart}

\usepackage[utf8]{inputenc}

\usepackage{amsmath,amssymb,amsfonts,amsthm,enumerate,mathtools,todonotes}

\usepackage[all]{xy}
\usepackage{color}
\usepackage{hyperref}
\theoremstyle{plain}
\newtheorem{theorem}{Theorem}[section]

\newtheorem{prop}[theorem]{Proposition}
\newtheorem{corollary}[theorem]{Corollary}

\newtheorem*{lem*}{Lemma}

\newtheorem*{thm*}{Theorem}

\theoremstyle{definition}
\newtheorem{definition}[theorem]{Definition}

\theoremstyle{remark}

\newcommand{\beq}{\begin{equation}}
\newcommand{\eeq}{\end{equation}}

\newcommand{\cG}{\mathcal{G}}

\newcommand{\G}{\mathcal{G}}

\newcommand{\Z}{\mathbb{Z}}

\newcommand{\N}{\mathbb{N}}

\newcommand{\Cs}{$C^*$-}

\renewcommand{\sp}{\mathrm{spec}}

\newcommand{\cH}{\mathcal{H}}

\newcounter{ContList}

\title[Paradoxicality for Graph Algebras]{Pure infiniteness and paradoxicality for graph $C^*$-algebras}
\author{Francesca Arici} 
\address{Institute for Mathematics, Astrophysics, and Particle Physics, Radboud University, Postbus 9010, 6500 GL Nijmegen,
The Netherlands}
\email{f.arici@math.ru.nl}

\author{Baukje Debets} 
\address{Section of Analysis, Department of Mathematics, KU Leuven,
Celestijnenlaan 200b - box 2400,
3001 Leuven, Belgium}
\email{baukje.debets@kuleuven.be}

\author{Karen R. Strung} 
\address{Institute for Mathematics, Astrophysics, and Particle Physics, Radboud University, Postbus 9010, 6500 GL Nijmegen,
The Netherlands}
\email{k.strung@math.ru.nl}


\date{\today}

\thanks{FA was partially supported by NWO under the VIDI-grant \mbox{016.133.326} and under the VENI-grant \mbox{016.192.237}. BD was partially supported by the European Research Council Consolidator Grant 614195. KS was funded by
Sonata 9 NCN grant 2015/17/D/ST1/02529 and a Radboud Excellence Initiative Postdoctoral Fellowship.}

\begin{document}
\begin{abstract}
We obtain necessary and sufficient conditions for pure infiniteness of the path groupoid $C^*$-algebra of a row-finite graph without sinks. In particular we show that for such a path groupoid $\mathcal{G}_E$, the properties of being essential principal and the existence of a basis of $(\mathcal{G}_E^a, 2, 1)$-paradoxical sets for the topology are not only sufficient, but also necessary.
\end{abstract}
\maketitle

\section{Introduction}

One of the most useful aspects of $C^*$-algebraic theory lies in its ability to interpret other mathematical or physical systems. Typically this involves input information from the system in question and an output $C^*$-algebra. Analysing the resulting $C^*$-algebra should then give information about the original system, and vice versa. The theory has been especially successful when the input is a directed graph. In this case, one constructs a universal $C^*$-algebra given by generators and relations determined by the vertex and edge sets of the graph. That one could construct a $C^*$-algebra for a directed graph was first shown in \cite{Watatani1982}, where directed graphs were associated to the incidence matrices of topological Markov chains and more general shifts of finite type; from these one constructs ${C}^*$-algebras using the methods of Cuntz and Krieger \cite{Cuntz1980}. Properties of a graph $C^*$-algebra $A$ are quite tractable in comparison to arbitrary $C^*$-algebras, thanks to the combinatorial techniques available for the underlying graph. For example, one can read off the ideal structure of $A$ directly from the graph \cite{BaHoReSz02}. Conversely, one can also use the structure of the $C^*$-algebra to identify information about the graph, for example, if $A$ is approximately finite (AF), then the graph cannot have loops \cite{KuPaRae98}.

In addition to the Cuntz--Krieger construction, one can start with a row-finite directed graph and first construct a locally compact \'etale groupoid coming from shift-equivalence of infinite paths \cite{KPRR:CK}. Then, using the construction of Renault \cite{Re80}, one can construct the (reduced) groupoid $C^*$-algebra.  Once again, it is possible to determine information about the graph $C^*$-algebra from its groupoid and vice versa. In this paper, we are interested in determining properties of the groupoid %$C^*$-algebra
 of a row-finite directed graph that imply that the corresponding \mbox{$C^*$-algebra} is purely infinite.

Pure infiniteness was introduced for simple $C^*$-algebras by Cuntz \cite{Cuntz1977a,Cuntz1977} as a $C^*$-algebraic analogue of the behaviour of type III von Neumann algebra factors. In the simple case, pure infiniteness implies a number of interesting structural properties. For example, if $A$ is purely infinite and simple then it cannot admit any non-trivial traces. This can be seen by the fact that the nonzero projections in a purely infinite $C^*$-algebras are always infinite, which is to say, they are always Murray--von Neumann equivalent to a proper subprojection \cite{Lin1991}.  Furthermore, a simple purely infinite $C^*$-algebra always has an abundant supply of projections---every nonzero hereditary $C^*$-subalgebra contains an infinite projection. In fact, for simple $C^*$-algebras this property is equivalent to being purely infinite and can be taken as the definition \cite{Zhang1990}.

In the non-simple case, it is not immediately clear what the appropriate notion of pure infiniteness should be. Cuntz's original formulation implies simplicity, so one is tempted to alternatively define a non-simple $C^*$-algebra to be purely infinite if every  nonzero hereditary $C^*$-subalgebra contains an infinite projection. Although this property, which we call Condition (IH) in the sequel, is interesting in its own right, it is not enough to guarantee some of the more important consequences of pure infiniteness that one sees in the simple case. In particular, one would like that a (simple or otherwise) purely infinite $C^*$-algebra cannot admit nonzero traces and that the tensor product of \emph{any} (simple or otherwise) $C^*$-algebra with the Cuntz algebra $\mathcal{O}_{\infty}$ is purely infinite.  A satisfactory definition was given in \cite{KiRo00}, which asks for a certain infinite condition for positive elements via Cuntz comparison. 

For the Cuntz--Krieger construction of non-simple graph $C^*$-algebras of row-finite graphs, sufficient conditions guaranteeing pure infiniteness are already known \cite{HoSz03}. However, if one forgets the Cuntz--Krieger construction and relies only on the groupoid model, such conditions are only known in the case of a simple graph $C^*$-algebras.  In this paper, we prove the following. 

\begin{thm*}[see Theorem {\bf \ref{thm:main}}]
Let $E$ be a row-finite directed graph without sinks. The following are equivalent:
\begin{enumerate}
\item $C^*(\cG_E)$ is purely infinite;
\item $\cG_E$ is essentially principal and for every finite path $\alpha$, the cylinder set $Z(\alpha)$ is paradoxical;
\item $E$ satisfies Conditions (K) and (DI);
\item $E$ satisfies Conditions (K) and (DL).
\end{enumerate}
\end{thm*}

Condition (K) was introduced in \cite{KPRR:CK} where it was used to generalise previous results from finite to infinite graphs. Conditions (I), (DI), and (DL) are conditions on the behaviour of paths in the graph and their relation to cylinder sets in the groupoid (see Definition~\ref{ConditionDefs}). The notion of paradoxicality of a cylinder set (Definition~\ref{def:paradoxical} below) comes from to the notion of paradoxicality for more general groupoids, introduced in \cite{paradoxical}, which in turn was inspired by similar definitions for group actions (see for example \cite{Rordam2012, Kerr2012}). Essentially, it means that we are able to see the peculiar Banach--Tarksi-like behaviour witnessed by infinite projections---that they can be decomposed into subprojections of equal size---at the level of the groupoid. 

The paper is organised as follows. In Section~\ref{Sect:Pre} we recall the constructions of graph $C^*$-algebras from a Cuntz--Krieger $E$-family and from the path groupoid, as well as some of the properties we will need in the sequel.

 In Section~\ref{Sect:PI} we sketch the ideas behind the characterisation of Condition (IH) in the groupoid model and discuss the difficulties of a straightforward generalisation to the non-simple case. Finally, we prove the main result, Theorem~\ref{thm:main}.

\subsubsection*{Acknowledgments} 
The authors wish to thank Christian B\"onicke, Kang Li, Bram Mesland, Bartosz Kwa\'{s}niewski, and Adam Rennie for helpful discussions. The main theorem of this note appears as part of the second author's Master's thesis, for which the first and third authors were supervisors. A special thanks goes out to Klaas Landsman, the third supervisor of the project, for his support.

\section{Preliminaries}
\label{Sect:Pre}
A \emph{directed graph} $E$ is a quadruple $(E^0, E^1, t, s)$ consisting of countable sets $E^0$ and $E^1$, called the set of \emph{vertices} and set of \emph{edges}, respectively, and two functions $t,s: E^1 \to E^0$ called the \emph{target} and \emph{source} maps. Note that some authors, such as  \cite{Rae05}, use the opposite convention for the maps $t$ and $s$. In what follows, we are only interested in graphs which are row-finite, as these are the graphs which allow for a groupoid model. A graph $E$ is called \emph{row-finite} if $s^{-1}(v)$ is a finite set for all $v\in E^0 $. A \emph{sink} is a vertex $v \in E^0$ such that $s^{-1}(v) = \emptyset$ and a $\emph{source}$ is a vertex $v \in E^{0}$ such that $t^{-1}(v) = \emptyset.$

\subsection{Cuntz--Krieger algebras}
The typical way of constructing a graph $C^*$-algebra is to realise it as a universal $C^*$-algebra from a \emph{Cuntz--Krieger E-family}.

\begin{definition}
\label{ckfamilie}
Let $E$ be a row-finite directed graph and $\cH$ a Hilbert space. A \emph{Cuntz--Krieger $E$-family $\left\lbrace S,P \right\rbrace$ on $\cH$} consists of a set $\left\lbrace P_v : v\in E^0 \right\rbrace$ of mutually orthogonal projections on $\cH$ and a set $\left\lbrace S_e : e\in E^1\right\rbrace$ of partial isometries on $\cH$, such that the following two conditions, called the Cuntz--Krieger relations, hold:
\begin{enumerate}
\item $S_e^*S_e=P_{t(e)}$ for all $e\in E^1$;
\item $P_v=\sum_{\left\lbrace e\in E^1 : s(e)=v\right\rbrace}S_eS_e^* $ whenever $v$ is not a sink.
\end{enumerate}
\end{definition} 
\noindent The \emph{graph $C^*$-algebra} $C^*(E)$ is then defined to be the universal $C^*$-algebra generated by $\left\lbrace S, P\right\rbrace$, subject to the Cuntz--Krieger relations. 

A \emph{finite path} in $E$ is a sequence $\mu=(\mu_1,...,\mu_k)$ of $k$ edges, $k \in \mathbb{N} \setminus \left\lbrace0\right\rbrace$, with $s(\mu_{i+1})=t(\mu_i)$ for $1\leq i\leq k-1 $. We extend the source and target maps by defining $s(\mu)=s(\mu_1)$ and $t(\mu)=t(\mu_k)$ and we denote the \emph{length} of $\mu$ by $|\mu|=k$.  If we denote by $E^k$ the set of paths of length $k$ in $E$, then the elements of $E^0$ (the vertices of $E$) can be regarded as paths of length $0$. A \emph{loop} (or \emph{cycle}) is a finite path $\mu$ such that $s(\mu)=t(\mu)$. We say that $\mu$ has an \emph{exit} if there exists $1\leq i \leq |\mu|$ and $e\in E^1$ such that $s(\alpha_i) =s(e)$ and $\alpha_i \neq e$. 

Define 
\[ E^* := \bigcup_{n\geq 0} E^n,\] the set of finite paths in $E$. If $\mu \in E^*$ is a finite path, then we denote by $S_{\mu}$ the element $S_{\mu_1} S_{\mu_2} \cdots S_{\mu_k} \in C^*(E)$, and we have
\[C^*(E)= \overline{\text{{span}}}\left\lbrace S_\mu S_\nu^* : \mu,\nu \in E^*, t(\mu)=t(\nu) \right\rbrace.\]

\subsection{The path groupoid}
We can also associate a $C^*$-algebra to a row-finite directed graph $E$ via a groupoid. Let $\mathcal{G}$ be a locally compact and Hausdorff groupoid with locally compact unit space $\mathcal{G}^{(0)}$. Denote the range and domain maps by $r, d : \mathcal{G} \to \mathcal{G}^{(0)}$. The ordered pair $(g, h) \in \cG \times \cG$ is composable if $d(g) = r(h)$ and if so the composition is denoted $gh$. The set of composable pairs is given by 
\[ \mathcal{G}^{(2)} = \left\lbrace (g, h) \in \mathcal{G} \times \mathcal{G} \mid d(g) = r(h) \right\rbrace.\]
  The inverse of $g \in \mathcal{G}$ is denoted $g^{-1}$. A groupoid is \emph{\'etale} if $r$ and $d$ are local homeomorphisms. In this case $\mathcal{G}^{(0)}$ is an open subset of $\mathcal{G}$ and there is a canonical Haar system is given by counting measures. An open subset $U \subset \mathcal{G}$ is called an \emph{open bisection} if both $d|_U$ and $r|_U$ are homeomorphisms onto their ranges.  If $\mathcal{G}$ is \'etale and $\mathcal{G}^{(0)}$ is totally disconnected, then the groupoid $\mathcal{G}$ is said to be \emph{ample}. Equivalently, an \'etale groupoid is ample if it has a basis of compact open bisections. 

Let $E$ be a row-finite directed graph. An \emph{infinite path} is an infinite sequence of  edges $x_1, x_2, \dots$ with the property that $s(x_{i+1})=t(x_i)$ for every $i \geq 1$. The \emph{infinite path space} is defined to be
\[E^{\infty}= \left \lbrace(x_1,x_2,...) :  x_i \in E^1 , t(x_i) =s(x_{i+1}) \,\, \forall i \geq 1 \right\rbrace.\]

The infinite path space is a subset of the product space $\prod_{i=1}^\infty E^1$ and thus inherits the product topology for which the \emph{cylinder sets} \[Z(\mu)=\left\lbrace x\in E^\infty : x_1=\mu_1, ..., x_{|\mu|}=\mu_{|\mu|}\right\rbrace, \] 
with $\mu \in E^*$, form a basis of open sets. Observe that the cylinder sets are also closed since  $\bigcup_{i=1}^{|\mu|} \pi^{-1}_i (E^1\setminus \left\lbrace \mu_i \right\rbrace)$ is open in $\prod_{i=1}^\infty E^1$ and
 \[E^\infty \setminus Z(\mu) = E^\infty \cap \left( \bigcup_{i=1}^{|\mu|} \pi^{-1}_i (E^1\setminus \left\lbrace \mu_i \right\rbrace)\right).\]  
The cylinder sets form a basis for a locally compact, $\sigma$-compact, totally disconnected, Hausdorff topology on $E^{\infty}$ \cite[Corollary~2.2]{KPRR:CK}.  

\begin{definition}
For a directed graph $E$, define the associated \emph{path groupoid} by \[\cG_E = \left\lbrace (x,k, y) \mid x,y \in E^{\infty},\ k \in \Z, \ \exists N \in \N \text{ with }  x_i=y_{i+k}\,\,  \forall  i\geq N \right\rbrace,\] with unit space $\cG^{(0)} \cong E^{\infty}$, and domain and range maps
\[ d,r : \cG \to \cG^{(0)}\]
given by $d((x,k,y)) = y$ and $r((x,k,y)) = x$. Composition and inverse are given by 
\[(x,k,y)(y,l,z) = (x,k+l,z), \qquad (x,k,y)^{-1} = (y,-k,x).\]
\end{definition}
Observe that $(x,k,y) \in \cG_E$ if and only if $x=wz$ and $y=vz$ where $w, v$ are finite paths whose lengths satisfy $|w|+ k =|v|$.

We could also describe the path groupoid by using the \emph{shift map} $\sigma : E^\infty \to E^\infty$ defined by $(\sigma x)_i= x_{i+1}$ for all $i \in \Z^+$. This shift map is a local homeomorphism and hence $(E^\infty, \sigma) $ is a one-sided subshift of finite type over the alphabet $E^1$.  The groupoid $\cG_E$ thus arises from equivalence with lag of infinite paths, that is,
\[ \cG_E =  \left\lbrace (x,k,y) \mid x,y \in E^\infty,\ k \in \Z,\  \exists N \in \N \text{ with }  \sigma^N (x)=\sigma^{N+k} (y) \right\rbrace. \]

The groupoid $\cG_E$ can be endowed with a topology with respect to which it is a second countable locally compact \'etale Hausdorff groupoid. A basis of this topology is given by

\[Z(\alpha, \beta)= \left\lbrace (x,k,y) \mid  x\in Z(\alpha),\ y\in Z(\beta),\ k=|\beta|-|\alpha|,\ x_i=y_{i+k} \text{ for } i>|\alpha| \right\rbrace.\]%\end{definition}
where $\alpha,\beta \in E^*$ are (possibly empty) paths with $t(\alpha)= t(\beta)$. It is not hard to check that $Z(\alpha, \beta)$ is a compact open bisection, hence $\G_E$ is ample.

To any locally compact \'etale groupoid $\mathcal{G}$, we associate its reduced groupoid $\mathrm{C}^*$-algebra as follows. Let $C_c(\mathcal{G})$ denote the (vector space of) compactly supported continuous functions on $\mathcal{G}$. For $f_1, f_2, f \in C_c(\mathcal{G})$ we define multiplication and involution by
\[ (f_1 \cdot f_2) (g) = \sum_ {h_1 h_2= g} f_1(h_1) f_2(h_2), \text{ for all } g \in \mathcal{G}, \]
and 
\[ f^*(g) = \overline{f(g^{-1})}. \]
With these operations $C_c(\mathcal{G})$ is a $^*$-algebra. For every $x \in \mathcal{G}^{(0)}$, let $\ell^2(d^{-1}(x))$ denote the Hilbert space of square-summable functions on $d^{-1}(x)$. From this we can define a $^*$-representation 
\[ \pi_x :  C_c(\mathcal{G}) \to \mathcal{B}(\ell^2(d^{-1}(x))) \]
by
\[ (\pi_x(f)\xi)(g) = \sum_{h_1 h_2 = g} f(h_1) \xi(h_2), \]
for $f \in C_c(\mathcal{G})$, $\xi \in \ell^2(d^{-1}(x))$, $g \in d^{-1}(x)$.  The reduced groupoid $C^*$-algebra, denoted $C^*_r(\mathcal{G})$ is the completion of $C_c(\mathcal{G})$ with respected to the norm
\[ \|f \| = \sup_{x \in \mathcal{G}^{(0)}} \| \pi_x(f) \|.\]

One may also define a full groupoid $C^*$-algebra $C^*(\mathcal{G})$. In the case that the groupoid is amenable, as is the case for the groupoid associated to a row-finite directed graph \cite[Theorem~4.2]{Pa02}, the two coincide.

Let $\G_E$ be the path groupoid of the row-finite directed graph $E$ with no sinks. It is easy to check that for two basis sets $Z(\alpha,\beta)$,  $Z(\gamma,\delta)$ their intersection is $Z(\alpha,\beta)$, $Z(\gamma,\delta)$ or the empty set. The assumption that $E$ has no sinks implies that for all $\alpha,\beta \in E^*$ with $t(\alpha)=t(\beta)$, the sets $Z(\alpha)$, $Z(\beta)$ and $Z(\alpha,\beta)$ are nonempty. Thus for any row-finite directed graph $E$ without sinks, $C^*(\cG_E)$ is generated by a Cuntz--Krieger $E$-family obtained by looking at characteristic functions on bisections of the path groupoid. Furthermore, as was mentioned above, by \cite[Theorem~4.2]{Pa02} the groupoid $\cG_E$ is always amenable. In summary, we have the following:

\begin{theorem}[{\cite[Section 4]{KPRR:CK}}]
\label{isomorphic algebras}
Let $E$ be a row-finite directed graph without sinks, and let $\cG_E$ be the corresponding path groupoid. Then $\cG_E$ is an ample amenable locally compact Hausdorff groupoid and $C^*(\cG_E) \cong C^*(E)$.
\end{theorem}

\section{Purely infinite path groupoid $C^*$-algebras} \label{Sect:PI}

 A projection $p$ in a $C^*$-algebra $A$ is said to be \emph{infinite} if it is Murray--von Neumann equivalent to a proper subprojection of itself, which is to say, there exists a partial isometry $v \in A$ such that $v^*v = p$ and $vv^* \lneqq p$. If $A$ is a simple unital $C^*$-algebra and $1_A$ is an infinite projection, then $A$ is called \emph{purely infinite}. The first examples of purely infinite simple $C^*$-algebras were the Cuntz algebras, $\mathcal{O}_n$, $n \in \mathbb{N}$, and $\mathcal{O}_{\infty}$. Purely infinite simple $C^*$-algebras boast a number of interesting properties: they are traceless, always have real rank zero (and hence always contain many projections), every nonzero hereditary $C^*$-subalgebra contains an infinite projection, every unitary can be approximated by a unitary of finite spectrum, and so  \cite{Zhang1990, Lin1991} (see also \cite[Proposition 4.1.1]{Rordam2002} for a further list). The Cuntz algebra $\mathcal{O}_{\infty}$ is particularly important among simple, separable, unital, purely infinite $C^*$-algebras. If $A$ is any simple, separable, nuclear, unital $C^*$-algebra, then $A \otimes \mathcal{O}_\infty$ is always purely infinite. Moreover if $A$ is a separable, nuclear, unital, purely infinite $C^*$-algebra, then $A \cong A \otimes \mathcal{O}_{\infty}$ \cite{Kirchberg2000}.
 
\subsection{Groupoid $C^*$-algebras and Property (IH)}

A groupoid $\G$ is called  \emph{topologically principal} (or \emph{essentially free}) if the set of points with trivial isotropy is dense in $\G^{(0)}$. We say that an \'etale groupoid $\G$ is \emph{locally contracting} if, for every nonempty open subset $U \subset \mathcal{G}^{(0)}$, there exist an open subset $V \subset U$ and an open bisection $S$ with $\overline{V} \subset d(S)$ and $d(\overline{V}S^{-1}) \subsetneqq V$.

 If $\G$ is a topologically principal \'etale groupoid which is locally contracting, then the reduced groupoid $C^*$-algebra $C^*_r(\G)$ has property (IH): every nonzero hereditary $C^*$-subalgebra of $C^*_r(\G)$ contains an infinite projection \cite[Proposition 2.4]{DLR97}. When a $C^*$-algebra $A$ is simple, property (IH) is equivalent to pure infiniteness \cite{Zhang1990}.
 
In the original statement of the theorem below, the graph $E$ is also assumed to be locally finite. This was required at the time for the groupoid to be amenable. Since this is no longer the case and local finiteness is not required elsewhere in their proof, we use the reformulation below.  See also  \cite[Proposition~5.3]{BaPaReSz00}, where the statement is proved in the Cuntz--Krieger model. 

\begin{theorem}[{\cite[Theorem 3.9]{KuPaRae98}}] \label{thm:IH} Let $E$ be a directed graph with no sinks.  Then $C^*(E)$ has property (IH) if and only if every vertex connects to a loop and every loop has an exit.
\end{theorem}

The proof amounts to showing that the path groupoid is topologically principal and locally contracting if and only if every vertex connects to a loop and every loop has an exit. For a unit $x \in \G_E^{(0)} = E^{\infty}$, one shows that non-trivial isotropy corresponds to eventual periodicity \cite[Lemma 3.2]{KuPaRae98}.  It follows that to show that $\G_E$ is topologically principal, it is enough to show that every vertex is the source of an aperiodic path, since this in turn implies that every basic set contains an aperiodic path. By \cite[Lemma 3.4]{KuPaRae98} this occurs exactly when every loop has an exit. This allows one to construct a path starting at a vertex $v$ that never passes through the same vertex twice, hence is aperiodic. Conversely, if there is a loop without an exit, then it is straightforward to find an eventually periodic path. The locally contracting condition is satisfied if every vertex connects to a vertex that has a return path with an exit. 

If there is a vertex which does \emph{not} connect to a loop, one shows that $C^*(\G_E)$ has a hereditary $C^*$-subalgebra which is approximately finite (AF). Similarly, if there is a loop that does not have an exit one can show that $C^*(\G_E)$ has a hereditary $C^*$-subalgebra that is Morita equivalent to a commutative $C^*$-algebra and hence does not have property (IH). In particular, $C^*(\G_E)$ does not have property (IH).

\subsection{Non-simple purely infinite $C^*$-algebras}
If $C^*(\G_E)$ is simple, which is the case if and only if every loop in $E$ has an exit and $E$ is cofinal (see \cite[Proposition 5.1]{BaPaReSz00}), then the above implies $C^*(\G_E)$ is purely infinite. However, for the non-simple case, we require a stronger condition. 

Extending the definition of pure infiniteness from simple $C^*$-algebras to non-simple $C^*$-algebras is a subtle matter. One needs to decide which key properties of pure infiniteness should hold in the non-simple case, such as $\mathcal{O}_{\infty}$-absorption and admitting no nonzero traces.  To do so, one needs to consider subequivalence of positive elements, as introduced by Cuntz in \cite{MR0467332}. Let $A$ be a $C^*$-algebra and $a, b \in A$ with $a, b \geq 0$. We say that $a$ is \emph{Cuntz-subequivalent} to $b$, written $a \precsim b$, if there exists a sequence $(r_n)_{n \in \mathbb{N}} \subset A$ such that $\| r_n b r_n^* - a \| \to 0$ as $n \to \infty$. More generally, if $a \in M_n(A)_+$ and $b \in M_m(A)_+$, then we define $a \precsim b$ if there exists a sequence of rectangular matrices $(r_n)_{n \in \mathbb{N}} \subset M_{n,m}(A)$  such that $\|r_n^*b r_n - a \| \to 0$ as $n \to \infty$. If $a, b \in A$ then we define the direct sum of $a$ and $b$ by
\[ a \oplus b = \left( \begin{array}{cc} a & 0 \\ 0 & b \end{array} \right) \in M_2(A).\]

 The definition below is due to Kirchberg and R{\o}rdam \cite[Definition 4.1, Theorem 4.16]{KiRo00}.

\begin{definition}
\label{def:KR_PI}
Let $A$ be a $C^*$-algebra. We say $A$ is \emph{purely infinite} if $a \oplus a \precsim a$ for every nonzero positive element $a \in A$. 
\end{definition}

This is already enough to imply that if $A$ is purely infinite, then $A$ admits no nonzero traces and that $A \otimes \mathcal{O}_\infty$ is always purely infinite. On the other hand, some other properties one might expect are not automatic. However, if $A$ is separable, nuclear, and has real rank zero then the situation begins to look a lot more like the simple case; for example $A$ is purely infinite if and only if $A \otimes \mathcal{O}_{\infty}$ is purely infinite, if and only if all nonzero projections in $A$ are properly infinite \cite{Pasnicu2007}. We will see below that real rank zero is automatic in the setting of graph algebras.

Let $E$ be a row-finite directed graph with no sinks. For all vertices $v,w \in E^0$, we write $v \geq w$ if there exists a finite path $\alpha \in E^*$ such that $s(\alpha)=v$ and $t(\alpha)=w$. Obviously for all $v\in E^0$ we have $v\geq v$ (letting $\alpha$ be the empty path) and for all $u,v,w\in E^0$ it follows that $u \geq v$ and $v \geq w$ imply $u \geq w$. Thus $\geq$ is a preorder. Note that it is not a partial order as there might be two vertices $v\neq w$ on a cycle such that $v\geq w \geq v$.

\begin{definition}
\label{maximal tail}
Let $M\subseteq E^0$ be a non-empty subset. Then we call $M$ a \emph{maximal tail} if the following three conditions hold:
\begin{enumerate}
\item If $v\in E^0$, $w\in M$, and $v\geq w$, then $v\in M$;
\item If $v \in M$ and $s^{-1}(v) \neq \emptyset$, then there exists $e\in E^1$ with $s(e)=v$ and $t(e) \in M$;
\item For every $v,w \in M$ there exists $y\in M$ such that $v \geq y$ and $w \geq y$.
\end{enumerate}
\end{definition}
Let $E$ be a row-finite directed graph. We define $V^2$ to be the set of vertices $v$ for which there are at least two distinct finite cycles based at $v$, that is,
\begin{eqnarray*}
V^2 &=& \{ v \in E^0 \mid \text{ there are cycles } \mu \neq  \nu \text{ with } t(\mu_i) = t(\nu_j) = v  \\
&&\text{ if and only if } i = |\mu| \text{ and } j = |\nu| \}. 
\end{eqnarray*}
Similarly, we define $V^1$ to be the set of vertices that lie exactly in one cycle, and $V^0$ to be the set of vertices $v$ for which there is no cycle based at $v$.
\begin{definition}[{\cite[Section~6]{KPRR:CK}}]
A graph $E$ \emph{satisfies Condition (K)} if $V^1 = \emptyset$, or, equivalently, if $E^0=V^0 \cup V^2$.
\end{definition}

Note that condition (K) generalises Condition (II) in \cite{Cuntz1977}. Indeed, a \emph{finite} directed graph satisfies condition (K) if and only if the associated incidence matrix satisfies condition (II).

For a row-finite directed graph  $E$ we have the following result, due to Hong and Szymanski. Note that it does not require that every loop in $E$ has an exit or that $E$ is cofinal; in particular, the result holds for non-simple graph $C^*$-algebras. 

\begin{theorem}[{\cite[Theorems 2.3, 2.5]{HoSz03}}] \label{thm:HoSz}
Let $E$ be a row-finite directed graph. Then the following are equivalent:
\begin{enumerate}
\item $C^*(E)$ is purely infinite;
\item $C^*(E)$ is purely infinite and has real rank zero;
\item all loops in each maximal tail $M$ have exits in $M$ and each vertex in every maximal tail of $M$ connects to a loop in $M$;
\item $E$ satisfies Condition (K) and each vertex in every maximal tail of $M$ connects to a loop in $M$.
\end{enumerate}
\end{theorem}

The goal of this section is to interpret the theorem above in terms of the path groupoid.  Unlike for graph $C^*$-algebras, for an arbitrary \'etale groupoid necessary and sufficient conditions for pure infiniteness are not known. This question was recently addressed in \cite{paradoxical}.  There, the authors establish a sufficient condition on an ample groupoid $\G$ that ensures pure infiniteness of the reduced \Cs algebra $C^*_r(\G)$ by extending the notion of paradoxical decompositions for actions of discrete groups on totally disconnected spaces to the setting of \'etale groupoids. Here, we define \emph{paradoxicality} for the path groupoid of a row-finite directed graph with no sinks.

\begin{definition}
\label{def:paradoxical}
 For a graph $E$ and a finite path $\mu \in E^*$, we say that the cylinder set $Z(\mu)$ is \emph{paradoxical} if there exist numbers $m,n \in \N$ and compact open bisections \[Z(\alpha_1, \beta_1), Z(\alpha_2, \beta_2) ,\dots ,Z(\alpha_n, \beta_n), Z(\gamma_1, \delta_1), Z(\gamma_2, \delta_2), \dots ,Z(\gamma_m, \delta_m),\] such that:
\[\bigcup_{i=1}^n Z(\beta_i) =\bigcup_{j=1}^m Z(\delta_i) = Z(\mu),\] 
 and the sets $\displaystyle Z(\alpha_i), Z(\gamma_j) \subseteq Z(\mu)$ are pairwise disjoint.
\end{definition}

In the language of \cite{paradoxical} this says that $Z(\mu)$ is ``$(\G^a, 2,1)$-paradoxical''. It is easy to see that for all $\mu \in E^*$, $Z(\mu)$ is paradoxical if and only if $Z(r(\mu))$ is paradoxical.

We will  need a stronger version of topological freeness. Let $\G$ be a locally compact groupoid. A subset $D \subset \G^{(0)}$ is called \emph{invariant} if for any $g \in \G$, $d(g) \in \G$ implies $r(g) \in \G$. Let $\G_D := \left\lbrace g \in \G \mid d(g) \in D\right\rbrace$. When $\G$ is \'etale and $D \subset \G^{(0)}$ is closed and invariant, then $\G_D$ is a closed  \'etale subgroupoid. We say $\G$ is \emph{essentially principal} if, for every closed invariant subset $D \subset \G^{(0)}$, the subgroupoid $\G_D$ is topologically principal.

\begin{prop}\label{prop:para_vetex}
Let $E$ be a row-finite graph without sinks. Suppose that $\G_E$ is essentially principal. Then $C^*(\cG_E)$ is purely infinite if the cylinder set $Z(v)$ is paradoxical for every $v \in E^0$.
\end{prop}

\begin{proof}
The path groupoid $\G_E$ is ample and by assumption essentially principal. Furthermore, since $\cG_E$ is amenable, \cite[Lemma 6.1]{BrClSi15} and \cite[Remark 6.2]{BrClSi15} imply that it is inner exact in the sense of \cite[Definition 3.5]{paradoxical}. Since $\left\lbrace Z(v) \mid v \in E^{0}\right\rbrace$ is a basis for the topology of $\G^{(0)} \cong E^{\infty}$, it  follows from \cite[Corollary 4.12]{paradoxical} that $C^*(\cG_E)$ is purely infinite if $Z(v)$ is paradoxical for every $v \in E^0$.
\end{proof}

We would like a condition on the graph $E$ that implies paradoxicality of $Z(v)$ for every $v \in V$.  To that end, we introduce the following  conditions.

\begin{definition} \label{ConditionDefs}
Let $E$ be a row-finite directed graph with no sinks.
\begin{enumerate} 
\item The graph $E$ satisfies Condition (I) if, for every $v \in E^0$, there exists a finite path $\alpha \in E^*$ with $s(\alpha) = v$ and $t(\alpha) \in V^2$.
\item The graph $E$ satisfies Condition (DI) if, for every $v \in E^0$, there exists a decomposition $Z(v) = \sqcup_{i=1}^n Z(\beta_i)$ such that for every $i = 1, \dots, n $ there exists a path $\alpha_i \in E^*$ with $t(\alpha_i) = t(\beta_i)$, $s(\alpha_i) = v$, and $\alpha_i$ passes through $V^2$.
\item The graph $E$ satisfies Condition (DL) if, for every $v \in E^0$, there exists a decomposition $Z(v) = \sqcup_{i=1}^n Z(\beta_i)$ such that for every $i = 1, \dots, n $ there exists a path $\alpha_i \in E^*$ with $t(\alpha_i) = t(\beta_i)$, $s(\alpha_i) = v$, and $\alpha_i$ passes through a loop.
\end{enumerate}
\end{definition}

Condition (I) was first defined in \cite{Cuntz1980} by Cuntz and Krieger for the Cuntz--Krieger algebras associated to an $n \times n$ $\left\lbrace0,1\right\rbrace$-matrix.  Note that (DI) implies (I) as well as (DI) implies (DL),  but neither converse need  hold. Let $E$ be a row-finite directed graph without sinks. We have already seen that Condition (K) holds whenever $C^*(\G_E) \cong C^*(E)$ is purely infinite.  We show in the sequel that (I), (DI), and (DL) are also necessary, and that when combined with Condition (K), either Condition (DI) or condition (DL) is enough to establish that $C^*(\G_E)$ is purely infinite. Condition (I), on the other hand, is not enough, even in the presence of Condition (K).

Condition (K) implies that $\G_E$ is essentially principal  \cite[Proposition 6.3]{KPRR:CK} and thus is required for our main result, Theorem~\ref{thm:main}.  In fact, Condition (K) is also necessary to ensure that for every $v \in E^0$, the cylinder set $Z(v)$ is paradoxical. Indeed, suppose that $(K)$ does not hold for $E$.  Then there exists a cycle $\mu$ in some $v$ with no other return path. If $Z(v)$ has a paradoxical decomposition, then \[x=\mu\mu\mu\dots \in Z(\beta_i)\cap Z(\delta_j) ,\]
 where \[\beta_i= \underbrace{\mu \dots \mu}_{k \text{ times}}\mu_1 \dots \mu_l \quad \text{ and }\quad \delta_j = \underbrace{\mu \dots \mu}_{n \text{ times}}\mu_1 \dots \mu_m,\] for some $k,l,n,m$. Then $\alpha_i, \gamma_j$ must be paths from $v$ to vertices on $\mu$, but these paths must be subpaths of $x$, for otherwise we would have two distinct return paths in $v$. Thus $Z(\alpha_i)\cap Z(\gamma_j) = Z(\mu_1 \dots \mu_p)\neq \emptyset$ for some $p \in \N$.  Hence $Z(v)$ cannot have a paradoxical decomposition. 
 
In view of this, let us examine Conditions (I), (DI) and (DL) in the presence of Condition (K).

 \begin{prop} 
 Let $E$ be a row-finite directed graph without sinks. Suppose that $E$ satisfies Condition (K). Then if the cylinder set $Z(v)$ is paradoxical for every $v \in E^0$, the graph $E$ satisfies Condition (I).
 \end{prop}
 
 \begin{proof}
Assume that the cylinder set $Z(v)$ is paradoxical for every $v \in E^0$ but that Condition (I) does not hold. Then there must be a vertex $v$ that is not connected to $V^2$, and because Condition (K) holds, $v$ is not connected to $V^1$ either. Following the proof of \cite[Theorem 3.9]{KuPaRae98} (see also the discussion following Theorem~\ref{thm:IH} in the previous subsection), let $H$ be the subgraph of $E$ formed by those vertices that can be reached from $v$. It is clear that $H$ has no loops and no exits.  By \cite[Theorem 6.6]{KPRR:CK}, there exists an isomorphism between the lattice of ideals in the groupoid  $C^*$ algebra $C^*(\cG_E)$ and the lattice of saturated subsets of $E^0$ (see \cite[Section 6]{KPRR:CK}). In particular, by \cite[Theorem 3.7]{KuPaRae98} and \cite[Proposition 2.1]{KuPaRae98}, $C^*(H)$ is a hereditary $C^*$-subalgebra of $C^*(\cG_E)$. Since $H$ has no loops it is AF, and so $C^*(\cG_E)$ cannot be purely infinite. Thus $E$ has no paradoxical decomposition in every vertex. Therefore Condition (I) must hold. \end{proof}

On the other hand, Condition (I) is not sufficient. Let $E$ be the following graph:
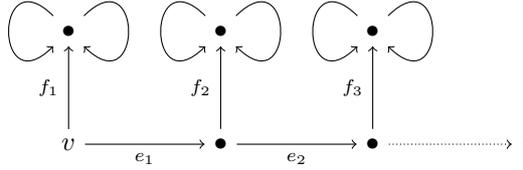
\begin{center}
\begin{tikzpicture}
\node (p1) at (0,-1.5) {$v$};
\node (p2) at (0,0) {$\bullet$};
\node (p3) at (2,-1.5) {$\bullet$};
\node (p4) at (2,0) {$\bullet$};
\node (p5) at (4,-1.5) {$\bullet$};
\node (p6) at (4,0) {$\bullet$};
\node (p7) at (6,-1.5) {$. $};

\draw[->] (p1) -- (p2) node [midway, left] (TextNode) { $\scriptstyle f_1$};
\draw[->] (p2) to [out=135,in=225,looseness=7] (p2) ;
\draw[->] (p2) to [out=45,in=315,looseness=7] (p2)  ;

\draw[->] (p3) -- (p4) node [midway, left] (TextNode) { $\scriptstyle f_2$};
\draw[->] (p4) to [out=135,in=225,looseness=7] (p4) ;
\draw[->] (p4) to [out=45,in=315,looseness=7] (p4)  ;

\draw[->] (p5) -- (p6) node [midway, left] (TextNode) { $\scriptstyle f_3$};
\draw[->] (p6) to [out=135,in=225,looseness=7] (p6) ;
\draw[->] (p6) to [out=45,in=315,looseness=7] (p6)  ;

\draw[->] (p1) -- (p3) node [midway, below, sloped] (TextNode) { $\scriptstyle e_1$};
\draw[  ->] (p3) -- (p5) node [midway, below, sloped] (TextNode) { $\scriptstyle e_2$};
\draw[ densely dotted, ->] (p5) -- (p7) ;
\end{tikzpicture}
\end{center}
Both Conditions (I) and (K) hold for this graph. Suppose now that $Z(v)$ has a paradoxical decomposition, where $v$ is the vertex on the bottom left (as denoted). Then \[\bar{e}= e_1e_2\cdots \in Z(\beta_i)\cap Z(\delta_j)\] for some $i$ and $j$, but then we must have $\beta_i = e_1\cdots e_n$ and $ \delta_j = e_1 \cdots e_m$ for some $n,m \in \N$. Thus $\alpha_i=e_1 \cdots e_n$ and $\beta_j =e_1 \cdots e_m$ as well, so $Z(\alpha_i)\cap Z(\gamma_j)\neq \emptyset$. Hence $Z(v)$ has no paradoxical decomposition, so (I) is not a sufficient condition for paradoxicality.

Condition (DI) is a strengthening of Condition (I). In the example above, we constructed a graph with a path containing only vertices in $V^0$ to show that (I) was not sufficient to paradoxical decompositions of the cylinder sets.  Thus, one might suppose that a suitable strengthening of (I) might be given by the property that  \emph{every} path must pass through $V^2$. However, all cylinder sets can be paradoxical in absence of this condition, as we see in the following example. Consider the graph $E$ pictured below.
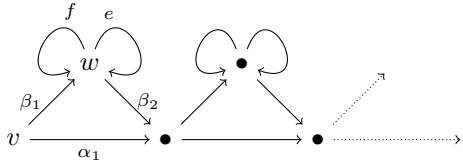
\begin{center}
\begin{tikzpicture}
\node (p1) at (0,-1) {$v$};
\node (p2) at (1,0) {$w$};
\node (p3) at (2,-1) {$\bullet$};
\node (p4) at (3,0) {$\bullet$};
\node (p5) at (4,-1) {$\bullet$};
\node (p6) at (5,0) {};
\node (p7) at (6,-1) {};

\draw[->] (p1) -- (p2) node [midway, left] (TextNode) { $\scriptstyle \beta_1$};
\draw[->] (p2) to [out=110,in=200,looseness=7] (p2)  node [above=16pt  ] (TextNode) { $\scriptstyle f$};
\draw[->] (p2) to [out=70,in=340,looseness=7] (p2)   node [above=16pt  ] (TextNode) { $\scriptstyle e$};

\draw[->] (p2) -- (p3) node [midway, right] (TextNode) { $\scriptstyle \beta_2$};

\draw[->] (p3) -- (p4) ;
\draw[->] (p4) to [out=110,in=200,looseness=7] (p4) ;
\draw[->] (p4) to [out=70,in=340,looseness=7] (p4)  ;
\draw[->] (p4) to (p5);

\draw[densely dotted, ->] (p5) -- (p6);

\draw[->] (p1) -- (p3) node [midway, below, sloped] (TextNode) { $\scriptstyle \alpha_1$};
\draw[  ->] (p3) -- (p5) ;
\draw[ densely dotted, ->] (p5) -- (p7) ;
\end{tikzpicture}
\end{center}
Observe that $E$ contains paths that do not pass through $V^2$. However $E$ satisfies (DL). Indeed, for every vertex $u \in E^0$, the set $\left\lbrace x\in E^{\leq \infty} \mid s(x)=u \right\rbrace$ is either isomorphic to $\left\lbrace x\in E^{\leq \infty} \mid s(x)=v \right\rbrace$ or to $\left\lbrace x\in E^{\leq \infty} \mid s(x)=w \right\rbrace$. Thus we need only show that there are paradoxical decompositions for $Z(v)$ and $Z(w)$, and these are easily seen to exist: For $Z(v)$ take
\[ Z( \beta_1 e ,\beta_1), Z(\alpha_1, \alpha_1) \quad \text{ and }\quad Z( \beta_1 f ,\beta_1), Z(\beta_1\beta_2, \alpha_1);\]
and for $Z(w)$ take
\[ Z( ee, e ), Z( ff ,f), Z(\beta_2, \beta_2) \quad \text{ and } \quad Z( ef, e ), Z( fe ,f), Z(e\beta_2, \beta_2).\]
  Note however, $E$ does satisfy (DI), and indeed, as the next proposition shows, this is precisely the condition we are after. 

\begin{prop} \label{prop:IIPara}
Let $E$ be a row-finite directed graph without sinks that satisfies (K). Then $E$ satisfies Condition (DI) if and only if $Z(v)$ is paradoxical for all $v\in E^0$.
\end{prop}

\begin{proof}
Suppose $E$ satisfies Condition (DI). There are decompositions $Z(\alpha_i^1, \beta_i)$ and $Z(\alpha_i^2, \beta_i)$ such that $Z(\alpha_i^k)\cap Z(\alpha_j^l) = \emptyset$ for all $i,j \leq n$ and $k,l =1,2$. Since all $\alpha_i$ pass through $V^2$ and for $w\in V^2$ we know that for all $m\in \N$ there exists $\mu_1, \dots, \mu_m \in E^*$ with $s(\mu_i)=t(\mu_i)=w$ and $Z(\mu_i)\cap Z(\mu_j) =\emptyset$ for all $i\neq j$.

For the other direction, suppose $E$ does not satisfy (DI). Then for all $n\in \N$ the set $P_n$ defined by
 \[ \left\lbrace \gamma \in E^n : s(\gamma)=v \text{ and } \nexists \alpha \text{ that passes through } V^2  \text{ with } s(\alpha)=v, t(\alpha)=t(\gamma)\right\rbrace\] is non-empty and for all $\mu \in P_{n+1}$, we have $\mu_{1}\dots\mu_n \in P_n$. Thus by Zorn's Lemma, there exists an infinite path $x\in E^\infty$ such that there is no path from $v$ via a vertex in $V^2$ to $t(x_n)$ for some $n\in \N$. However, that means that \[M = \left\lbrace u \in E^0 : u\geq t(x_n) \text{ for some } n\in \N \right\rbrace \] is a maximal tail, where $v$ does not connect to a loop in $M$. 
Therefore the algebra \mbox{$C^*(E) \cong C^*(\cG_E)$} is not purely infinite, and hence by \cite[Corollary 4.12]{paradoxical} the cylinder set $Z(v)$ is not paradoxical for all $v\in E^0$.
\end{proof}

The next proposition will allow us to examine the role of Condition (DL).

\begin{prop}
Let $E$ be a row-finite directed graph without sinks, and let $v\in E^0$. The following are equivalent.
\begin{enumerate}
\item There exists a decomposition $Z(v) = \sqcup_{i=1}^n Z(\beta_i)$ such that for every $i = 1, \dots, n $ there exists a path $\alpha_i \in E^*$ with $t(\alpha_i) = t(\beta_i)$, $s(\alpha_i) = v$, and $\alpha_i$ passes through a loop.
\item If $M$ is a maximal tail containing $v$, then there is a loop in $M$ that $v$ is connected to.
\end{enumerate}
\end{prop}

\begin{proof}
Suppose (1) holds and let $M$ be a maximal tail with $v\in M$. Then by Definition \ref{maximal tail} (2), there exists an infinite path $x\in E^\infty$ with $s(x)=v$ and $t(x_n)\in M$ for all $n\in \N$.
We know that \[Z(v) = \bigsqcup_{i=1}^n Z(\beta_i),\] so $x \in Z(\beta_i)$ for some $1\leq i\leq n$.
Hence $t(\beta_i)=t(\alpha_i)\in M$, and then by Definition \ref{maximal tail} (1), every vertex on $\alpha_i$ is in $M$ and $\alpha_i$ passes through a loop. So that loop is in $M$ as well, and lastly $v$ connects to that loop. Thus (1) implies (2). 

For the converse, (1) does not hold, that is, there is no decomposition $Z(v) = \sqcup_{i=1}^n Z(\beta_i)$ satisfying the requirements of Definition~\ref{ConditionDefs} (3). Using Zorn's lemma, once again, there exists an infinite path $x\in E^\infty$ with $s(x)=v$, such that there is no path from $v$ via a loop to $t(x_n)$ for some $n\in \N$. Then \[M = \left\lbrace u \in E^0 : u\geq t(x_n) \text{ for some } n\in \N \right\rbrace \] is again a maximal tail, where $v$ does not connect to a loop. Thus (2) implies (1).
\end{proof}

\begin{corollary} \label{cor:KIII-pi}
Let $E$ be a row-finite directed graph without sinks. The graph $E$ satisfies Conditions (K) and (DL) if and only if $C^*(\G_E)$ is purely infinite.
\end{corollary}

\begin{proof} By the previous proposition $(DL)$ is equivalent to the property that every vertex in every maximal tail of $M$ connects to a loop in $M$.  Thus the result follows from the equivalence of (1) and (4) of Theorem~\ref{thm:HoSz}, 
\end{proof}

As we observed earlier, Condition (DI) implies Condition (DL), but the reverse does not necessarily hold. However, in the presence of Condition (K), they are equivalent.

\begin{corollary}\label{cor:II iff III}
 Let $E$ be a row-finite directed graph without sinks that satisfies Condition (K). Then $E$ satisfies Condition (DI) if and only if $E$ satisfies (DL).
\end{corollary}

\begin{proof}
We only need to show that Condition (DL) implies (DI), but this is immediate since (K) implies that at any vertex with a loop, there is a second loop.
\end{proof}

Finally, as a summary of the above, we come to the main theorem.

\begin{theorem}
\label{thm:main}
Let $E$ be a row-finite directed graph without sinks. The following are equivalent:
\begin{enumerate}
\item $C^*(\cG_E)$ is purely infinite;
\item $\cG_E$ is essentially principal and for every finite path $\alpha$, the cylinder set $Z(\alpha)$ is paradoxical;
\item $E$ satisfies Conditions (K) and (DI);
\item $E$ satisfies Conditions (K) and (DL).
\setcounter{ContList}{\value{enumi}}
\end{enumerate}
\end{theorem}

\begin{proof}
The implication (2) implies (1) is given by Proposition~\ref{prop:para_vetex}. The equivalence of (3) and (4) is Corollary~\ref{cor:II iff III}.  If (3) holds, then Condition (K) implies that $\mathcal{G}_E$ is essentially principal and since Condition (DI) also holds, the cylinder set of any finite path is paradoxical by Proposition~\ref{prop:IIPara}. Thus (3) implies (2). Finally, the equivalence of (4) and (1) follows directly from Corollary \ref{cor:KIII-pi}.
\end{proof}

Observe that if $C^*(\cG_E)$ satisfies any of the equivalent conditions of Theorem~\ref{thm:main}, then it automatically has real rank zero by Theorem~\ref{thm:HoSz}. Thus each of (1)--(4) is also equivalent to the following properties:
\begin{enumerate}
\setcounter{enumi}{\value{ContList}}
\item $C^*(\cG_E) \otimes \mathcal{O}_{\infty} \cong C^*(\cG_E)$;
\item $C^*(\cG_E)$ is strongly purely infinite (see \cite[Definition 5.1]{MR1906257});
\item  every nonzero hereditary $C^*$-subalgebra in any quotient of $C^*(\cG_E)$ contains an infinite projection \cite[Proposition 2.11]{Pasnicu2007}.
 \end{enumerate}

Unlike the Hong--Szymanski result of Theorem~\ref{thm:HoSz}, we require the assumption that $E$ has no sinks. This comes from the fact that the path groupoid is only constructed for a row-finite graph without sinks. However, one can deal with graphs with sinks by adding tails, a technique introduced in  \cite{BaPaReSz00}. Given a graph $E$  with sinks, we denote by $F$ the graph obtained by adding a tail to every sink. Then $C^*(\mathcal{G}_E) \cong C^*(E)$ is isomorphic to a full corner in $C^*(F)$  \cite{BaPaReSz00}. In particular, $C^*(E)$ is purely infinite if and only if $C^*(F)$ is purely infinite.

Finally, it also worth noting that the equivalence of (1) and (2) in the main theorem is a stronger result than what one has for an arbitrary ample amenable \'etale groupoid $\cG$. In that case, if $\cG$ is essentially principal and has a basis for the topology consisting of  $(\G^a, 2, 1)$-paradoxical sets, then $C^*(\cG)$ is (strongly) purely infinite by  \cite[Corollary 4.12]{paradoxical}. However, unlike for the path groupoid, the converse is not known.

\end{document}